\documentclass[12pt]{amsart}
\usepackage{amsmath,amssymb,amsthm}
\usepackage{tikz}
\usepackage{xcolor}
\usepackage[margin=1.15in]{geometry}
\usepackage[colorlinks=true,citecolor=blue,linkcolor=blue,urlcolor=blue]{hyperref}

\theoremstyle{plain}
\newtheorem{theorem}{Theorem}
\newtheorem{lemma}{Lemma}[section]
\newtheorem{proposition}{Proposition}[section]
\newtheorem{corollary}{Corollary}[section]
\newtheorem*{conjecture*}{Conjecture}

\theoremstyle{definition}
\newtheorem{definition}{Definition}[section]

\theoremstyle{remark}
\newtheorem{remark}{Remark}[section]

\newcommand{\N}{\mathbb{N}}
\newcommand{\Z}{\mathbb{Z}}
\newcommand{\A}{\mathcal{A}}
\newcommand{\B}{\mathcal{B}}
\newcommand{\Lo}{\mathcal{L}}
\newcommand{\core}{\operatorname{core}}
\newcommand{\wt}{\operatorname{wt}}
\newcommand{\Irr}{\operatorname{Irr}}

\renewcommand{\epsilon}{\varepsilon}

\title[Zero-free columns in character tables of symmetric groups]
      {Zero-free columns in the character tables of\\ symmetric groups}

\author[Defant]{Colin Defant}
\author[Hariharan]{Sidharth Hariharan}
\author[Lau]{Kenny Lau}
\author[Ono]{Ken Ono}

\address{Axiom Math, 124 University Avenue, Palo Alto, CA 94301}
\email{colin@axiommath.ai}
\email{kenny@axiommath.ai}
\email{ken@axiommath.ai}

\address{Dept. Mathematical Sciences,
Carnegie Mellon University,
Pittsburgh, PA 15213}
\email{shariha2@andrew.cmu.edu}


\begin{document}

\begin{abstract}
The rows and columns of the character table of the symmetric group $S_n$ are both naturally indexed by partitions of
$n$.  Let $D(n)$ denote the number of conjugacy classes of $S_n$ whose
column contains no zero entry.  The identity column is always zero-free,
so $D(n)\geq 1$.  It is known that $D(n)\ll n^2$.  We prove that $D(n)\ll n^{3/4}$.  Second, we prove for almost all positive integers $n$ that $D(n)\ll_B n^{1/2}(\log n)^B$ for every $B>5/6$, with a quantitative bound for the exceptional set, using work of Matom\"aki and Radziwi\l\l.  Finally, we offer a
heuristic supporting our conjecture that $D(n)\ll_{\varepsilon}
n^{\varepsilon}$. 
AxiomProver formalized the results in this paper in Lean assuming preexisting literature.
\end{abstract}

\maketitle

\section{Introduction}\label{sec:intro}

Character tables are among the most useful invariants of finite groups.  If $G$ is finite, then a character table is a (necessarily square) matrix $[\chi(C)]$, with the rows indexed by the irreducible characters $\chi$ of $G$, the columns indexed by the conjugacy classes $C$ of $G$, and $\chi(C)$ the value $\chi(g)\in\mathbb{C}$ shared by all $g\in C$.

For the symmetric group $S_n$, the character table has a concrete combinatorial form~\cite{JK}.  The conjugacy
classes are indexed by partitions $\mu\vdash n$, recording the cycle
lengths of a permutation.  The irreducible characters are also indexed
by partitions $\lambda\vdash n$, through the Specht modules.  We write
$\chi^{\lambda}(\mu)$ for the value of the irreducible character indexed
by $\lambda$ on the conjugacy class of cycle type $\mu$.  Thus, if $p(n)$ is the parition function, then the character table of $S_n$ is the square matrix
\[
  \mathcal{C}_n=\bigl[\chi^{\lambda}(\mu)\bigr]_{\lambda,\mu\vdash n},
\]
with $p(n)$ rows and $p(n)$ columns.  Hardy~and~Ramanujan~\cite{HR} proved the asymptotic
\begin{equation}\label{eq:HR}
  p(n)\sim\frac{1}{4n\sqrt{3}}\,e^{\pi\sqrt{2n/3}}.
\end{equation}

Zeros in these character tables have been studied from several different
perspectives.  Let $\Irr(S_n)$ be the set of irreducible characters of
$S_n$.  Miller \cite{Mi} proved that
\begin{equation}\label{eq:miller}
  \lim_{n\to\infty}
  \frac{\#\{(\chi,g)\in \Irr(S_n)\times S_n:\ \chi(g)=0\}}
       {p(n)\,n!}=1.
\end{equation}
Since $g$ is drawn uniformly from the group, the conjugacy class of $g$
is drawn with probability proportional to its size; in this sense
Miller's model weights conjugacy classes by their sizes.

A different
model, considered by Miller--Scheinerman~\cite{MS} and by
Peluse--Soundararajan \cite{PS}, weights the entries of the character
table uniformly.  In other words, they consider the problem where one
chooses both $\lambda$ and $\mu$ uniformly from the partitions of $n$. To be precise, we let
\[
  Z(n):=\#\{(\lambda,\mu):\ \lambda,\mu\vdash n,\ \chi^{\lambda}(\mu)=0\}
\]
denote the number of zero entries in the character table of $S_n$.  Thus, the full character table has $p(n)^2$ entries.  Miller and Scheinerman conjecture \cite{MS}, on the basis of large-scale Monte Carlo experiments, that
\[
  Z(n)\sim\frac{2p(n)^2}{\log n}.
\]
At present, no asymptotic formula is known for $Z(n)$, and the best
general upper bound for $Z(n)$ is the trivial estimate
$Z(n)\leq p(n)^2$.

On the other hand, Peluse and Soundararajan \cite{PS} proved that the
zeros produced by ``core-partition criteria,'' which can be thought of
as trivial zeros, already have the conjectured order of magnitude.
Following \cite{PS}, for $t\in\N$, let $\mathcal{H}_t(\lambda)$ be the number of hook lengths of $\lambda$ that are multiples of $t$, and let $\mathcal{P}_t(\mu):=\tfrac1t\sum_{t\mid\mu_j}\mu_j$.  Thus, $t\mathcal{P}_t(\mu)$ is the sum of the parts of $\mu$ divisible by $t$.  A pair $(\lambda,\mu)$ with $\chi^{\lambda}(\mu)=0$ is called a zero of \emph{type I} if $\lambda$ is a $\mu_1$-core, where $\mu_1$ is the largest part of $\mu$; a zero of \emph{type II} if $\lambda$ is a $t$-core for some part size $t$ of $\mu$; and a zero of \emph{type III} if Stanley's criterion applies (that is, if $\mathcal{P}_t(\mu)>\mathcal{H}_t(\lambda)$ for some $t$).  That each of these conditions forces $\chi^\lambda(\mu)=0$ follows from the Murnaghan--Nakayama rule. These families are nested:
\[
  \{\text{type I zeros}\}\subseteq\{\text{type II zeros}\}
  \subseteq\{\text{type III zeros}\}
  \subseteq\{(\lambda,\mu):\lambda,\mu\vdash n,\ \chi^{\lambda}(\mu)=0\}.
\]
If $Z_{\mathrm{I}}(n)$, $Z_{\mathrm{II}}(n)$,
$Z_{\mathrm{III}}(n)$ are the respective cardinalities, then Peluse and Soundararajan proved that
\[
  Z_{\mathrm{III}}(n)
  =p(n)^2\left(\frac{2}{\log n}
  +O\!\left(\frac{(\log\log n)^2}{(\log n)^2}\right)\right).
\]
Since these are genuine zeros, this gives the unconditional lower bound
\[
  Z(n)\geq p(n)^2\left(\frac{2}{\log n}
  +O\!\left(\frac{(\log\log n)^2}{(\log n)^2}\right)\right),
\]
which matches the conjectured asymptotic for $Z(n)$.  Thus the known
lower bounds for the number of zeros in the full table are presently
much stronger than the known upper bounds.

In this paper, we consider a complementary problem.  Instead of counting all zero entries, we count the number of columns that avoid zeros altogether.  To make this precise, a conjugacy
class, or the corresponding column of $\mathcal{C}_n$, is called
\emph{zero-free} if every entry in that column is nonzero.  For
$n\geq 1$, define
\begin{equation}\label{eq:Ddef}
  D(n):=\#\{\mu\vdash n:\ \chi^{\lambda}(\mu)\neq 0
  \text{ for every }\lambda\vdash n\}.
\end{equation}
The identity column, corresponding to $\mu=(1^n)$, is always zero-free
because $\chi^{\lambda}(1^n)=\dim S^{\lambda}>0$.  Therefore, the
assertion that every non-identity column contains a zero is precisely
the assertion that $D(n)=1$.

Duro \cite{Duro} studied this condition in connection with the
generalized Knutson index.  He observed that if every non-trivial column
of a finite group character table contains a zero, then the generalized
Knutson index agrees with the Knutson index \cite[Proposition~2.3]{Duro}.
For symmetric groups, he reported the initial values
\[
  1,\ 5,\ 6,\ 8,\ 9,\ 10,\ 12,\ 14,\ 17,\ 21,\ 28,\ 30,\ 32,\ 34,\ 36,\
  37,\ 38,\ \ldots
\]
of $n$ for which $D(n)=1$ \cite[Question~4.17]{Duro}.  Our computations
(described in Section~\ref{sec:appendix}) confirm this list and extend
it: the integers $n\leq 43$ with $D(n)=1$ are precisely
\[
  1,5,6,8,9,10,12,14,17,21,28,30,32,34,36,37,38,40,42.
\]
For comparison, we have\footnote{The computation is described in
Section~\ref{sec:appendix}.}
\[
  D(2)=D(3)=D(4)=2,\quad\ldots,\quad D(7)=4,\quad\ldots,\quad
  D(42)=1,\quad D(43)=7.
\]

Motivated by Duro's work and its connection to the computation of
Knutson indices, we consider the problem of deriving upper bounds for
$D(n)$.  It follows from \cite[Proposition~4.15]{Duro} that
\begin{equation}\label{eq:duro-bound}
  D(n)\ll n^2.
\end{equation}
In light of \eqref{eq:HR}, Duro's work already implies that zero-free
columns form an exponentially negligible fraction of all columns of the
character table.  We prove two results that improve the exponent in
\eqref{eq:duro-bound}, one for all $n$ and one for almost all $n$.

The main idea is that Duro's reduction confines zero-free columns to
cycle types of the form $(3^a,2^b,1^c)$ with $b$ even.  The
Murnaghan--Nakayama rule then turns the condition of being zero-free
into two independent core-distance restrictions: the number of
$2$-cycles is bounded by the distance from $n$ to the nearest
$2$-core size, while the number of $3$-cycles is bounded by the distance
from $n$ to the nearest $3$-core size.  The former restriction leads to
a triangular-number approximation problem, and the latter restriction
leads to an approximation problem by Loeschian numbers, or equivalently
by norms in the Eisenstein integers.

Our first theorem gives a pointwise bound.

\begin{theorem}\label{thm:pointwise}
For positive integers $n$,we have $D(n)\ll n^{3/4}$.
\end{theorem}

Our second theorem, which relies on work of Matom\"aki and
Radziwi\l\l\ \cite{MR}, improves the ``$3$-cycle bound'' for almost all
$n$.

\begin{theorem}\label{thm:almostall}
For all $B>5/6$, there exists a constant $C_B>0$ such that if $\epsilon>0$ and $X\geq 3$, then
\[
  \#\bigl\{1\leq n\leq X:\ D(n)>C_B\,n^{1/2}(\log n)^B\bigr\}
  \ll_{B,\epsilon}X(\log X)^{-\frac{1}{2}(B-\frac56)+\varepsilon}.
\]
In particular, $D(n)\ll_B n^{1/2}(\log n)^B$ for almost all $n$.
\end{theorem}

The on-average exponent $1/2$ in Theorem~\ref{thm:almostall} suggests that the pointwise exponent $3/4$ in Theorem~\ref{thm:pointwise} is not optimal.  We also present a heuristic in Section~\ref{sec:heuristic} that suggests  $D(n)\ll_{\varepsilon}n^{\varepsilon}$ for all $n$.  We note, however, that $3/4=1/2+1/4$ is the natural limit of the core-obstruction method used here: the $2$-core distance is genuinely of order $n^{1/2}$ for a positive proportion of $n$ (consecutive triangular numbers of a given parity are spaced $\asymp\sqrt{n}$ apart), while any pointwise improvement of the $3$-core distance bound $s_3(n)\ll n^{1/4}$ of Lemma~\ref{lem:s3} would improve upon the classical Bambah--Chowla bound \cite{BC} for gaps between sums of two squares, an exponent that has resisted improvement since 1947.  See also Remark~\ref{rem:typeIII}.  These observations suggest that our theorems are close to the limits of their respective methods.

The paper is organized as follows.
Section~\ref{sec:cores} recalls the necessary partition theory and the
Murnaghan--Nakayama vanishing criterion for character table entries.
Section~\ref{sec:duro} proves Duro's structural reduction in the form
needed here, which involves the combinatorial properties of $2$-core
and $3$-core partitions.  Section~\ref{sec:distances} gives the
$2$-core and $3$-core distance estimates, which are then used in
Section~\ref{sec:proof1} to prove Theorem~\ref{thm:pointwise}.
Section~\ref{sec:almostall} states the Matom\"aki--Radziwi\l\l\ input
and proves Theorem~\ref{thm:almostall}.  In
Section~\ref{sec:heuristic}, we give a heuristic that supports our
conjecture that $D(n)\ll_{\varepsilon}n^{\varepsilon}$.  In the final
section we describe the role of AI and the role of AxiomProver in the
production and verification of these proofs in Lean, and we offer
remarks concerning the computation of $D(n)$.

\section*{Acknowledgements}\noindent The authors thank Jesse Thorner for comments on an earlier version of this manuscript.

\section{Partitions, cores, and a vanishing criterion}\label{sec:cores}

A partition $\lambda$ of $n$, written $\lambda\vdash n$, is a
nonincreasing sequence of positive integers
$\lambda=(\lambda_1\geq\lambda_2\geq\cdots\geq\lambda_r>0)$ with
$\sum_i\lambda_i=n$.  We identify $\lambda$ with its Young diagram.

A \emph{rim hook} of length $t$ is a connected border strip of $t$
boxes whose removal leaves the Young diagram of a partition and that
contains no $2\times 2$ square.  Figure~\ref{fig:rimhook} shows an
example.  A partition is a \emph{$t$-core} if no rim hook of length $t$
can be removed from it.  Equivalently, no hook length in its Young
diagram is divisible by $t$; see \cite[Chapter~2]{JK} or \cite{Ol}.

\begin{figure}[ht]
\centering
\begin{tikzpicture}[scale=0.55]
  \fill[black!15] (3,-1) rectangle (5,0);   
  \fill[black!15] (2,-2) rectangle (4,-1);  
  \foreach \j in {0,...,4} {\draw (\j,0) rectangle ++(1,-1);}
  \foreach \j in {0,...,3} {\draw (\j,-1) rectangle ++(1,-1);}
  \foreach \j in {0,...,1} {\draw (\j,-2) rectangle ++(1,-1);}
  \draw (0,-3) rectangle ++(1,-1);
\end{tikzpicture}
\caption{The Young diagram of $\lambda=(5,4,2,1)$ with a rim hook of
length $4$ and height $1$ shaded.  Its removal leaves the partition
$(3,2,2,1)$.}
\label{fig:rimhook}
\end{figure}
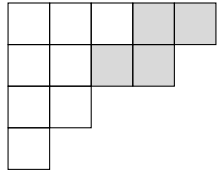

Starting from a partition $\lambda$, repeatedly removing rim hooks of
length $t$ eventually produces a $t$-core.  The final $t$-core is
independent of all choices and is denoted $\core_t(\lambda)$.  The
integer
\[
  \wt_t(\lambda):=\frac{|\lambda|-|\core_t(\lambda)|}{t}
\]
is the \emph{$t$-weight} of $\lambda$.  We record for later use the
classical fact (see \cite[\S 2.7]{JK} or \cite{Ol}) that the boxes of
$\lambda$ whose hook lengths are divisible by $t$ are in bijection with
the boxes of the $t$-quotient of $\lambda$; consequently, in the
notation of Section~\ref{sec:intro},
\begin{equation}\label{eq:hooks-weight}
  \mathcal{H}_t(\lambda)=\wt_t(\lambda).
\end{equation}

Let $\mathrm{RH}(\lambda,t)$ be the set of rim hooks in $\lambda$ of length $t$, and let $\operatorname{ht}(h)$ be the height of $h$, i.e.\ one less than its vertical span.  With this notation, the (recursive form of the) Murnaghan--Nakayama rule \cite[Theorem~2.4.7]{JK} states that if $t=\mu_i$ is the length of a chosen part of $\mu$, then
\[
  \chi^{\lambda}(\mu)
  =\sum_{h\in\mathrm{RH}(\lambda,t)}(-1)^{\operatorname{ht}(h)}
   \chi^{\lambda\setminus h}(\mu\setminus\mu_i).
\]
This allows us to compute $\chi^{\lambda}(\mu)$ recursively; note that the value of the iterated expansion is independent of the order in which the parts of $\mu$ are processed.

We shall use the following consequence of the Murnaghan--Nakayama rule.
It is the mechanism behind all the columnwise vanishing estimates in
this paper.

\begin{lemma}[Core-removal cutoff]\label{lem:cutoff}
Let $t\geq 2$, $n\geq 1$, and $s\geq 0$ be integers.  Suppose that there exists a
$t$-core partition of $n-st$.  Then every conjugacy class of $S_n$
whose cycle type contains more than $s$ parts equal to $t$ has a zero
in its column.
\end{lemma}

\begin{proof}
Let $\gamma$ be a $t$-core of size $n-st$.  We first construct a
partition $\lambda\vdash n$ whose $t$-core is $\gamma$ and whose
$t$-weight is $s$.  Add $st$ boxes to the end of the first row of
$\gamma$.  In other words, if $\gamma=(\gamma_1,\gamma_2,\ldots)$, put
\[
  \lambda=(\gamma_1+st,\gamma_2,\gamma_3,\ldots).
\]
Removing $s$ horizontal rim hooks of length $t$ from the right end of
the first row returns $\gamma$.  Since $\gamma$ is a $t$-core and the
$t$-core is independent of the removal choices,
we have $\core_t(\lambda)=\gamma$ and $\wt_t(\lambda)=s$.

Now let $\mu\vdash n$ contain $r>s$ parts equal to $t$.  In applying
the Murnaghan--Nakayama rule, order the parts of $\mu$ so that these
$r$ parts equal to $t$ are removed first.  A nonzero summand would
require a sequence of $r$ successive rim-hook removals of length $t$
starting from $\lambda$.  But $\lambda$ has $t$-weight $s$, so after
at most $s$ removals of rim hooks of length $t$ one reaches its
$t$-core, where no rim hook of length $t$ remains.  Hence there is no
such removal sequence.  The Murnaghan--Nakayama sum is empty, and
therefore $\chi^{\lambda}(\mu)=0$.
\end{proof}

We also need the following standard existence theorem for cores.

\begin{proposition}[Existence of $t$-cores]\label{prop:tcore}
If $t\geq 4$, then every nonnegative integer is the size of at least
one $t$-core partition.
\end{proposition}

\begin{proof}
This is the standard positivity theorem for $t$-core partitions when
$t\geq 4$.  It follows from the work of Granville--Ono; see
\cite{GO,On}.
\end{proof}

\section{Duro's reduction to 2- and 3-cycles}\label{sec:duro}

We next prove the structural reduction for zero-free columns.  This is
Duro's Proposition~4.15, with the short proof included for
completeness.

\begin{proposition}[Duro's reduction]\label{prop:duro}
Let $n\geq 3$, and let $\mu\vdash n$ be a partition whose column in the
character table is zero-free.  Then
\[
  \mu=(3^a,2^b,1^c)
\]
for some $a,b,c\geq 0$, and $b$ is even.
\end{proposition}

\begin{proof}
Suppose first that $\mu$ has a part $t\geq 4$.  By
Proposition~\ref{prop:tcore}, there is a $t$-core partition of $n$.
Lemma~\ref{lem:cutoff}, applied with $s=0$, shows that every conjugacy
class containing a part of size $t$ has a zero in its column.  This
contradicts the assumption that $\mu$ is zero-free.  Hence, every part
of $\mu$ is $1$, $2$, or $3$.

It remains to prove that the number $b$ of parts equal to $2$ is even.
A permutation of cycle type $(3^a,2^b,1^c)$ has sign $(-1)^b$, since a
$2$-cycle is odd while a $3$-cycle is even.  Thus if $b$ is odd, the
class is an odd conjugacy class.

For every $n\geq 3$, there exists a self-conjugate partition of $n$.
If $n=2m+1$ is odd, take $(m+1,1^m)$.  If $n=2m$ is even and
$m\geq 2$, take $(m,2,1^{m-2})$; for $m=2$ this is $(2,2)$.  These
partitions are equal to their transposes.

Let $\lambda$ be self-conjugate.  The standard identity
$\chi^{\lambda'}=\chi^{\lambda}\otimes\mathrm{sgn}$ therefore gives
$\chi^{\lambda}=\chi^{\lambda}\otimes\mathrm{sgn}$.  Evaluating at an
odd permutation $g$ yields
\[
  \chi^{\lambda}(g)=\mathrm{sgn}(g)\chi^{\lambda}(g)
  =-\chi^{\lambda}(g),
\]
so $\chi^{\lambda}(g)=0$.  Hence an odd conjugacy class cannot be
zero-free.  Therefore $b$ is even.
\end{proof}

\begin{remark}\label{rem:n2}
The case $n=2$ is exceptional for the parity argument: the
transposition column of $S_2$ is zero-free.  This has no asymptotic
significance, and all estimates below include small $n$ after enlarging
the absolute implied constants.
\end{remark}

\section{The 2-core and 3-core distances}\label{sec:distances}

Proposition~\ref{prop:duro} reduces the problem to bounding the
possible exponents $a$ and $b$ in a zero-free cycle type
$(3^a,2^b,1^c)$.  We do this by measuring how far $n$ is from a
$2$-core size and a $3$-core size.

\subsection{The 2-core distance}

The $2$-cores are especially simple.

\begin{lemma}\label{lem:2core}
A nonnegative integer $m$ is the size of a $2$-core partition if and
only if $m=u(u+1)/2$ for some integer $u\geq 0$.
\end{lemma}

\begin{proof}
We use the equivalent hook-length definition of a $2$-core.  Suppose
that $\lambda=(\lambda_1,\lambda_2,\ldots,\lambda_r)$ is a $2$-core,
and set $\lambda_{r+1}=0$.  First, if $\lambda_i-\lambda_{i+1}\geq 2$
for some $i$, then the box $(i,\lambda_i-1)$ has exactly one box to its
right and no box below it.  Its hook length is therefore $2$, which is
impossible in a $2$-core.  Hence $\lambda_i-\lambda_{i+1}\leq 1$ for
every $i$.

Second, no two consecutive row lengths can be equal.  Indeed, suppose a
maximal block of equal row lengths has at least two rows, and let rows
$j-1$ and $j$ be the last two rows in that block.  The box at the end
of row $j-1$ has exactly one box below it and no box to its right, so
its hook length is $2$.  This is impossible in a $2$-core.  Thus all
consecutive row lengths differ by exactly $1$, and the partition has
staircase shape $(r,r-1,\ldots,2,1)$.

Conversely, in the staircase $(r,r-1,\ldots,1)$ every hook length is
odd: the hook at position $(i,j)$ has length $2(r-i-j+1)+1$.  Thus no
hook length is divisible by $2$, and the staircase is a $2$-core.  Its
size is $r(r+1)/2$.
\end{proof}

\begin{definition}\label{def:s2}
For $n\geq 1$, define
\[
  s_2(n):=\min\Bigl\{s\geq 0:\ n-2s=\frac{u(u+1)}{2}
  \text{ for some integer } u\geq 0\Bigr\}.
\]
\end{definition}

\begin{lemma}\label{lem:s2}
For all $n\geq 1$, we have $s_2(n)\ll n^{1/2}$.
\end{lemma}

\begin{proof}
The finitely many cases $n<5$ are immediate.  Assume $n\geq 5$ and put
$N=8n+1$.  The condition that $n-2s$ be triangular is equivalent to
$8(n-2s)+1=(2u+1)^2$ for some $u\geq 0$.  Equivalently, there exists an odd positive integer $q$ such that $N-16s=q^2$.

The integer $N=8n+1$ is congruent to either $1$ or $9$ modulo $16$.
The odd square classes modulo $16$ are exactly $1$ and $9$.  Therefore
the admissible positive odd integers $q$ satisfying
$q^2\equiv N\pmod{16}$ lie in two residue classes modulo $8$, and the
gaps between consecutive admissible $q$'s are at most $6$.  Hence we
may choose such a $q$ with
\[
\sqrt{N}-6<q\leq\sqrt{N}.
\]
Set $s=(N-q^2)/16$.  The congruence makes $s$ a nonnegative integer,
and $N-16s=q^2$ implies that $n-2s$ is triangular.  Thus
$s_2(n)\leq s$, and
\[
  s_2(n)\leq\frac{N-(\sqrt{N}-6)^2}{16}
  =\frac{12\sqrt{N}-36}{16}\ll\sqrt{n}. \qedhere
\]
\end{proof}

\begin{corollary}\label{cor:bbound}
If $\mu=(3^a,2^b,1^c)\vdash n$ is a partition whose column in the
character table is zero-free, then $b\leq s_2(n)$.  In particular, we
have $b\ll n^{1/2}$.
\end{corollary}

\begin{proof}
By Lemma~\ref{lem:2core} and the definition of $s_2(n)$, there is a
$2$-core partition of $n-2s_2(n)$.  If $b>s_2(n)$, then
Lemma~\ref{lem:cutoff}, with $t=2$ and $s=s_2(n)$, gives a partition
$\lambda\vdash n$ such that $\chi^{\lambda}(\mu)=0$.  This contradicts
the assumption that $\mu$ is zero-free.  Therefore $b\leq s_2(n)$, and
the final estimate follows from Lemma~\ref{lem:s2}.
\end{proof}

\subsection{The 3-core distance}

Let
\begin{equation}
\label{eqn:Lo_def}
  \Lo:=\{x^2+xy+y^2:\ x,y\in\Z\}\cap\N
\end{equation}
be the set of positive Loeschian numbers.  Equivalently, $\Lo$ is the
set of positive norms from the Eisenstein integers.

\begin{lemma}[The 3-core criterion]\label{lem:3core}
For $m\geq 0$, there exists a $3$-core partition of $m$ if and only if
$3m+1\in\Lo$.
\end{lemma}

\begin{proof}
Let $c_3(m)$ denote the number of $3$-core partitions of $m$.  The
classical formula is
\[
  c_3(m)=\sum_{d\mid 3m+1}\left(\frac{d}{3}\right),
\]
where $(\tfrac{\cdot}{3})$ is the quadratic character modulo $3$; see
\cite{GO}.  Since $3m+1\equiv 1\pmod 3$, the integer $3m+1$ is not
divisible by $3$.  The divisor sum is multiplicative as a function of
$3m+1$.  Its local factor at a prime $p\equiv 1\pmod 3$ is $e+1>0$,
where $e=v_p(3m+1)$.  Its local factor at a prime $p\equiv 2\pmod 3$ is
$1-1+1-\cdots+(-1)^e$, which is positive if $e$ is even and is zero if
$e$ is odd.  Hence $c_3(m)>0$ if and only if every prime
$p\equiv 2\pmod 3$ divides $3m+1$ to even exponent.

The same parity condition on primes $p\equiv 2\pmod 3$ is the standard
prime-factorization criterion for representation by the Eisenstein norm
form $x^2+xy+y^2$; see, for example, \cite{Ma}.  Therefore $c_3(m)>0$
if and only if $3m+1\in\Lo$.
\end{proof}

\begin{definition}\label{def:s3}
For $n\geq 1$, define
\[
  s_3(n):=\min\{s\geq 0:\ n-3s\geq 0\text{ and }n-3s
  \text{ has a $3$-core partition}\}.
\]
Equivalently, by Lemma~\ref{lem:3core}, we have
$s_3(n)=\min\{s\geq 0:\ 3(n-3s)+1\in\Lo\}$.
\end{definition}

\begin{lemma}\label{lem:s3}
For all $n\geq 1$, we have $s_3(n)\ll n^{1/4}$.
\end{lemma}

\begin{proof}
It is enough to prove the estimate for all sufficiently large $n$,
since finitely many small values can be absorbed into the implied
constant.  Put $N=3n+1$.  Then $N\equiv 1,4,$ or $7\pmod 9$.  Choose
$r\in\{1,2,4\}$ such that $r^2\equiv N\pmod 9$.  Let
$T=\lfloor N^{1/4}\rfloor$.

Assume first that $N\geq 2401$.  Then $N-27T^2>0$, so the interval
\[
  I=\bigl[\sqrt{N-27T^2},\ \sqrt{N}\bigr]
\]
is well-defined.  Its length is greater than $9$.  Indeed, since
$T\geq N^{1/4}-1$, we have $T^2\geq\sqrt{N}-2N^{1/4}+1$.  Thus
\[
  \sqrt{N}-\sqrt{N-27T^2}
  =\frac{27T^2}{\sqrt{N}+\sqrt{N-27T^2}}
  \geq\frac{27(\sqrt{N}-2N^{1/4}+1)}{2\sqrt{N}}.
\]
Writing $x=N^{1/4}$, the last expression is
\[
  \frac{27}{2}-\frac{27}{x}+\frac{27}{2x^2},
\]
which is increasing for $x>1$ and is greater than $9$ at $x=7$.  Since
$N\geq 2401=7^4$, the claim follows.

Every interval of length greater than $9$ contains at least nine
consecutive integers, hence an integer in any prescribed residue class
modulo $9$.  Therefore $I$ contains an integer $A$ satisfying
$A\equiv r\pmod 9$.  Then
\[
  N-27T^2\leq A^2\leq N,\qquad A^2\equiv N\pmod 9.
\]
Define
\[
  u=\sqrt{\frac{N-A^2}{27}},\qquad y=\lfloor u\rfloor,
\]
and set $M=A^2+27y^2$.  Since $0\leq u\leq T$, we have $M\leq N$.
Moreover,
\[
  N-M=27(u^2-y^2)=27(u-y)(u+y)<27(2T+1)\ll N^{1/4}.
\]
Also $M\equiv A^2\equiv N\pmod 9$.  Finally, $M$ is Loeschian, because
\[
  M=A^2+27y^2=(A-3y)^2+(A-3y)(6y)+(6y)^2.
\]

Since $M\leq N$ and $M\equiv N\pmod 9$, there is an integer $s\geq 0$
such that $M=N-9s$.  The estimate above gives $s\ll N^{1/4}$.  But
\[
  M=N-9s=3(n-3s)+1\in\Lo.
\]
By Lemma~\ref{lem:3core}, $n-3s$ has a $3$-core partition.  Thus $s$ is
admissible in the definition of $s_3(n)$, and
\[
  s_3(n)\leq s\ll N^{1/4}\ll n^{1/4}
\]
for all $N\geq 2401$.  For the remaining finitely many $N<2401$, the
estimate follows after increasing the absolute implied constant.  This
proves the lemma.
\end{proof}

\begin{corollary}\label{cor:abound}
If $\mu=(3^a,2^b,1^c)\vdash n$ is a partition whose column in the
character table is zero-free, then $a\leq s_3(n)$.  In particular, we
have $a\ll n^{1/4}$.
\end{corollary}

\begin{proof}
By the definition of $s_3(n)$, there is a $3$-core partition of
$n-3s_3(n)$.  If $a>s_3(n)$, then Lemma~\ref{lem:cutoff}, with $t=3$
and $s=s_3(n)$, gives a partition $\lambda\vdash n$ such that
$\chi^{\lambda}(\mu)=0$.  This contradicts the assumption that $\mu$ is
zero-free.  Therefore $a\leq s_3(n)$, and the final estimate follows
from Lemma~\ref{lem:s3}.
\end{proof}

\section{Proof of Theorem~\ref{thm:pointwise}}\label{sec:proof1}

The cases $n<3$ are absorbed by the absolute implied constant, so
assume $n\geq 3$.  Let $\mu\vdash n$ be zero-free.  By
Proposition~\ref{prop:duro}, $\mu=(3^a,2^b,1^c)$ with $b$ even.  By
Corollaries~\ref{cor:bbound} and~\ref{cor:abound}, we have
\[
  0\leq a\leq s_3(n)\ll n^{1/4},\qquad
  0\leq b\leq s_2(n)\ll n^{1/2}.
\]
For each pair $(a,b)$, the value of $c$ is forced by $c=n-3a-2b$.
Counting all such pairs (the parity condition on $b$ would save a
further factor of $2$, which we do not need) gives
\[
  D(n)\leq(s_3(n)+1)(s_2(n)+1)\ll n^{3/4}. \qedhere
\]

\section{The almost-all refinement}\label{sec:almostall}

We now prove Theorem~\ref{thm:almostall} using work of Matom\"aki and
Radziwi\l\l\ \cite{MR}.  We state only the special form needed here.  A
subset $\mathcal{N}\subseteq\N$ is called \emph{multiplicative} if,
whenever $(m,n)=1$, we have
\[
  mn\in\mathcal{N}\iff m\in\mathcal{N}\text{ and }n\in\mathcal{N}.
\]
For such a set define
\[
  \delta(\mathcal{N};X):=\prod_{\substack{p\leq X\\ p\notin\mathcal{N}}}
  \left(1-\frac{1}{p}\right).
\]

\begin{theorem}[Matom\"aki--Radziwi\l\l\ gap theorem]\label{thm:MR}
Let $\mathcal{N}\subseteq\N$ be multiplicative.  Suppose that there is
a constant $\alpha>0$ such that, uniformly for $2\leq w\leq z$, one has
\begin{equation}\label{eq:sieve-cond}
  \sum_{\substack{w<p\leq z\\ p\in\mathcal{N}}}\frac{1}{p}
  \geq\alpha\sum_{w<p\leq z}\frac{1}{p}-O\!\left(\frac{1}{\log w}\right).
\end{equation}
Let $1\leq u_1<u_2<u_3<\cdots$ be the increasing sequence of elements of
$\mathcal{N}$.  Then, for every $1\leq\gamma<3/2$, we have
\[
  \sum_{u_i\leq X}(u_{i+1}-u_i)^{\gamma}
  \ll_{\alpha,\gamma}X\,\delta(\mathcal{N};X)^{1-\gamma}.
\]
\end{theorem}

\begin{proof}
Our assumptions imply that \eqref{eq:sieve-cond} holds for all
$2\leq w\leq z\leq X^{\alpha}$.  Therefore, \cite[Corollary~1.2(ii)]{MR} holds.
\end{proof}

\subsection{Multiplicative Loeschian subsets in fixed residue classes}

The $3$-core condition requires Loeschian numbers congruent to $3n+1$
modulo $9$.  Since $3n+1$ is always congruent to $1$, $4$, or $7$
modulo $9$, these are the only residue classes needed.

Define
\[
\mathcal{A}=\{m\in\N\colon v_3(m)=0,~\textup{$6\mid v_p(m)$ every prime $p\not\equiv 1\pmod{9}$}\}.
\]
Thus primes congruent to $1$ modulo $9$ may occur to arbitrary exponent,
while every other prime must occur to exponent divisible by $6$, except
that the prime $3$ is not allowed to occur. Recall that $\Lo=\{x^2+xy+y^2:\ x,y\in\Z\}\cap\N$. 

\begin{lemma}\label{lem:A}
The set $\mathcal{A}$ has the following properties.
\begin{enumerate}
\item[(i)] $\A$ is multiplicative.
\item[(ii)] $\A\subseteq\{m\in\Lo:\ m\equiv 1\pmod 9\}$.
\item[(iii)] $\delta(\A;X)\asymp(\log X)^{-5/6}$.
\item[(iv)] If $1\leq a_1<a_2<\cdots$ are the elements of $\A$, then
for every $1\leq\gamma<3/2$,
\[
  \sum_{a_i\leq X}(a_{i+1}-a_i)^{\gamma}
  \ll_{\gamma}X(\log X)^{\frac{5}{6}(\gamma-1)}.
\]
\end{enumerate}
\end{lemma}

\begin{proof}
The definition of $\A$ is prime-by-prime.  If $(m,n)=1$, then each
prime exponent in $mn$ is the corresponding exponent in exactly one of
$m$ or $n$.  Hence the defining conditions hold for $mn$ if and only if
they hold for both $m$ and $n$.  This proves (i).

For (ii), let $m\in\A$.  If $p\equiv 2\pmod 3$, then
$p\not\equiv 1\pmod 9$, so $v_p(m)$ is a multiple of $6$ and in
particular is even.  By the prime-factorization criterion for the
Eisenstein norm form, $m\in\Lo$.  Also $3\nmid m$.  If
$p\equiv 1\pmod 9$, then $p^e\equiv 1\pmod 9$ for every $e\geq 0$.  If
$p\not\equiv 1\pmod 9$ and $p\neq 3$, then $p^6\equiv 1\pmod 9$, and
$v_p(m)$ is a multiple of $6$.  Therefore every prime-power factor of
$m$ is congruent to $1$ modulo $9$, so $m\equiv 1\pmod 9$.

The primes belonging to $\A$ are exactly the primes
$p\equiv 1\pmod 9$.  Therefore
\[
  \delta(\A;X)=\prod_{\substack{p\leq X\\ p\not\equiv 1\ (9)}}
  \left(1-\frac{1}{p}\right).
\]

Mertens' theorem for arithmetic progressions \cite{Wi} gives
\[
  \sum_{\substack{p\leq X\\ p\equiv 1\ (9)}}\frac{1}{p}
  =\frac{1}{6}\log\log X+O(1),
  \qquad\text{hence}\qquad
  \sum_{\substack{p\leq X\\ p\not\equiv 1\ (9)}}\frac{1}{p}
  =\frac{5}{6}\log\log X+O(1).
\]
Taking logarithms of the Euler product for $\delta(\A;X)$ proves
(iii).

Mertens' theorem for arithmetic progressions also gives, uniformly for
$2\leq w\leq z$,
\[
  \sum_{\substack{w<p\leq z\\ p\equiv 1\ (9)}}\frac{1}{p}
  =\frac{1}{6}\sum_{w<p\leq z}\frac{1}{p}
  +O\!\left(\frac{1}{\log w}\right).
\]
Hence the hypothesis \eqref{eq:sieve-cond} of Theorem~\ref{thm:MR}
holds for $\A$, for instance with any fixed $\alpha<1/6$.  Combining
Theorem~\ref{thm:MR} with (iii) gives
\[
  \sum_{a_i\leq X}(a_{i+1}-a_i)^{\gamma}
  \ll_{\gamma}X\,\delta(\A;X)^{1-\gamma}
  \ll_{\gamma}X(\log X)^{\frac{5}{6}(\gamma-1)},
\]
which proves (iv).
\end{proof}

Note that $1,4,7\in\Lo$ since
\[
  1=1^2+1\cdot 0+0^2,\qquad 4=2^2+2\cdot 0+0^2,\qquad
  7=2^2+2\cdot 1+1^2.
\]
For $r\in\{1,4,7\}$ define $\B_r:=r\A$.  Because Loeschian numbers
are norms and norms multiply, $\B_r\subseteq\Lo$.  By
Lemma~\ref{lem:A}, every element of $\B_r$ is congruent to $r$ modulo
$9$.  Hence
\[
  \B_r\subseteq\{m\in\Lo:\ m\equiv r\pmod 9\}.
\]
If $1\leq b^{(r)}_1<b^{(r)}_2<\cdots$ are the elements of $\B_r$, then
Lemma~\ref{lem:A}(iv) implies that, for every $1\leq\gamma<3/2$,
\begin{equation}\label{eq:Br-gaps}
  \sum_{b^{(r)}_i\leq X}
  \bigl(b^{(r)}_{i+1}-b^{(r)}_i\bigr)^{\gamma}
  \ll_{\gamma,r}X(\log X)^{\frac{5}{6}(\gamma-1)}.
\end{equation}
Indeed, $\B_r$ is obtained from $\A$ by multiplication by the fixed
integer $r$.

\subsection{A tail estimate for the 3-core distance}

\begin{proposition}\label{prop:tail}
Let $1<\gamma<3/2$.  For $X\geq 3$ and every integer $S\geq 1$,
\[
  \#\{1\leq n\leq X:\ s_3(n)>S\}
  \ll_{\gamma}S+X(\log X)^{\frac{5}{6}(\gamma-1)}S^{1-\gamma}.
\]
\end{proposition}

\begin{proof}
Let $1\leq n\leq X$ and put $N=3n+1$.  Then $N\leq 3X+1$ and
$N\equiv 1,4,$ or $7\pmod 9$.  Let $r\in\{1,4,7\}$ be the residue class
of $N$ modulo $9$.

If $s_3(n)>S$, then none of the integers
$N,N-9,N-18,\ldots,N-9S$ belongs to $\Lo$.  Since every element of
$\B_r$ lying in $[N-9S,N]$ is congruent to $r\equiv N\pmod 9$, hence of
the form $N-9s$ with $0\leq s\leq S$, the interval $[N-9S,N]$ contains
no element of $\B_r$.

The initial range $N\leq 9S+10$ contributes $O(S)$ possible integers
$n$.  Outside this range, $N-9S>10$, so there is at least one element
of $\B_r$ below $N-9S$; for example, $r\in\B_r$ and $r\leq 7$.  Let
\[
  b^{(r)}_i<N<b^{(r)}_{i+1}
\]
be the consecutive elements of $\B_r$ surrounding $N$.  Since
$[N-9S,N]$ contains no element of $\B_r$, we have
$b^{(r)}_i<N-9S<N<b^{(r)}_{i+1}$, and therefore
\[
  g^{(r)}_i:=b^{(r)}_{i+1}-b^{(r)}_i>9S.
\]

For a fixed gap of length $g^{(r)}_i$, the number of integers
$N\equiv r\pmod 9$ lying in that gap is at most $g^{(r)}_i/9+1\ll
g^{(r)}_i$.  Each such $N$ corresponds to at most one integer $n$,
namely $n=(N-1)/3$.  Hence
\[
  \#\{1\leq n\leq X:\ s_3(n)>S\}
  \ll S+\sum_{r\in\{1,4,7\}}
  \sum_{\substack{b^{(r)}_i\leq 3X+1\\ g^{(r)}_i>9S}}g^{(r)}_i.
\]
Since $g^{(r)}_i>9S$ and $\gamma>1$,
\[
  g^{(r)}_i\leq(9S)^{1-\gamma}\bigl(g^{(r)}_i\bigr)^{\gamma}.
\]
Using \eqref{eq:Br-gaps}, we obtain
\[
  \sum_{\substack{b^{(r)}_i\leq 3X+1\\ g^{(r)}_i>9S}}g^{(r)}_i
  \ll_{\gamma}S^{1-\gamma}X(\log X)^{\frac{5}{6}(\gamma-1)}.
\]
Summing over $r=1,4,7$ proves the proposition.
\end{proof}

\begin{proof}[Proof of Theorem~\ref{thm:almostall}]
Fix $B>5/6$ and $1<\gamma<3/2$.  Put $S=\lfloor(\log X)^B\rfloor$.
Proposition~\ref{prop:tail} gives
\begin{align*}
  \#\{1\leq n\leq X:\ s_3(n)>(\log X)^B\}
  &\ll_{B,\gamma}(\log X)^B+X(\log X)^{(\frac56-B)(\gamma-1)} \\
  &\ll_{B,\gamma}X(\log X)^{-(B-\frac56)(\gamma-1)}.
\end{align*}
We now convert this into a statement with $(\log n)^B$ rather than
$(\log X)^B$.  Apart from the $O(X^{1/2})$ integers $n<X^{1/2}$, every
$n\leq X$ satisfies $\log n\geq\tfrac12\log X$.  Thus, for
$n\in[X^{1/2},X]$,
\[
  (\log X)^B\leq 2^B(\log n)^B.
\]
It follows that
\[
  \#\{1\leq n\leq X:\ s_3(n)>2^B(\log n)^B\}
  \ll_{B,\gamma}X(\log X)^{-(B-\frac56)(\gamma-1)}.
\]

For $n$ outside this exceptional set, repeat the counting argument from
the proof of Theorem~\ref{thm:pointwise}.  A zero-free column has type
$(3^a,2^b,1^c)$ with
\[
  a\leq s_3(n)\ll_B(\log n)^B,\qquad b\leq s_2(n)\ll n^{1/2},
\]
and the value of $c$ is determined by $c=n-3a-2b$.  Hence, outside a
set of at most $O_{B,\gamma}\bigl(X(\log X)^{-(B-5/6)(\gamma-1)}\bigr)$
integers $n\leq X$,
\[
  D(n)\leq(s_3(n)+1)(s_2(n)+1)\ll_B n^{1/2}(\log n)^B.
\]
This is the asserted estimate.
\end{proof}

\section{A heuristic and conjecture}\label{sec:heuristic}

The following discussion offers a heuristic and conjecture regarding
the size of $D(n)$.
The theorems above reduce the search for zero-free columns to the
sparse family
\[
  \mu=(3^a,2^b,1^c),\qquad
  0\leq a\ll n^{1/4},\qquad 0\leq b\ll n^{1/2},\qquad
  b\equiv 0\!\!\pmod 2.
\]
This leaves $O(n^{3/4})$ possible columns.  Duro's exact problem asks
which, if any, of these columns besides $(1^n)$ are actually zero-free.

A naive random model would suggest that very few survive.  For a fixed
non-identity column $\mu$, the values $\chi^{\lambda}(\mu)$ as
$\lambda$ varies are constrained integers, but after the obvious core
obstructions have been removed there is no evident reason that all
$p(n)$ of them should avoid zero.  Even a very small independent
probability of vanishing among the rows would make the probability that
the entire column is zero-free exponentially small in a power of
$p(n)$.

A more arithmetic version of the same heuristic is as follows.  Write
$\mu=(3^a,2^b,1^c)$.  The Murnaghan--Nakayama formula computes
$\chi^{\lambda}(\mu)$ as a signed sum over ways of removing $a$ rim
hooks of length $3$, $b$ rim hooks of length $2$, and then $c$ rim
hooks of length $1$.  The separate $2$-core and $3$-core arguments only
show that removal is impossible when $a$ or $b$ is too large.  But when
both $a$ and $b$ are present, the signed sums should usually have many
competing terms of both signs.  It is therefore reasonable to expect
many cancellations to zero unless $(a,b)$ lies in a very exceptional
set.

\begin{remark}\label{rem:typeIII}
The identity \eqref{eq:hooks-weight} shows that our vanishing input is
exactly as strong as the full type III (Stanley) criterion of
Section~\ref{sec:intro}.  Indeed, for $\mu=(3^a,2^b,1^c)$ the only
relevant moduli are $t=2$ and $t=3$, with $\mathcal{P}_2(\mu)=b$ and
$\mathcal{P}_3(\mu)=a$, and by \eqref{eq:hooks-weight},
\[
  \min_{\lambda\vdash n}\mathcal{H}_t(\lambda)
  =\min_{\lambda\vdash n}\wt_t(\lambda)=s_t(n).
\]
Thus the set of columns not eliminated by any type III zero is
precisely the rectangle $0\leq a\leq s_3(n)$, $0\leq b\leq s_2(n)$
counted in the proofs of Theorems~\ref{thm:pointwise}
and~\ref{thm:almostall}.  Any improvement in the exponents of these
theorems therefore requires exhibiting, in each surviving column, a
zero that is \emph{not} of type III---exactly the class of zeros that
current techniques cannot count \cite{PS}.
\end{remark}

One possible conjectural formulation is the following.

\begin{conjecture*}
For all $\varepsilon>0$, there exists a constant $c(\varepsilon)>0$
such that $D(n)\leq c(\varepsilon)n^{\varepsilon}$.
\end{conjecture*}

An even stronger possibility, compatible with Duro's computations, is
that $D(n)=1$ for infinitely many $n$, or perhaps for almost all $n$.
We do not know how to prove either statement.  The main obstacle is
that current core-partition arguments give only impossibility of
rim-hook removals, whereas the conjectural improvement would need to
exploit cancellation inside the Murnaghan--Nakayama sums or a deeper
columnwise vanishing criterion.

\section{Appendix}\label{sec:appendix}

\subsection{AxiomProver produced a formal Certificate}

The paper was written with human-AI collaboration; the outline of the proofs were crafted by the human authors, and then completed with AI assistance.
In particular, AxiomProver, an AI system currently under development by Axiom Math, generated a formal certificate for Theorem~\ref{thm:pointwise} and Theorem~\ref{thm:almostall}, relative to existing literature. Namely, these results were formalized and proved in Lean assuming the Murnaghan-Nakayama rule, a deep result of Matom\"aki--Radziwi\l\l\, and standard facts about 2-core and 3-core partitions.  For explicit details, see 
\begin{center}
\url{https://github.com/AxiomMath/ZeroFree}
\end{center}
This directory contains a formal challenge file containing the statements of the two theorems, which can be mechanically verified using the Comparator tool in Lean.

\subsection{Computation of \texorpdfstring{$D(n)$}{D(n)}}
\label{subsec:computation}

We include the details behind the values displayed in the
introduction.  For fixed $n$, Proposition~\ref{prop:duro} reduces the
search for zero-free columns to the candidate cycle types
$(3^a,2^b,1^c)$ with $b$ even; Corollaries~\ref{cor:bbound}
and~\ref{cor:abound} further restrict to $a\leq s_3(n)$ and
$b\leq s_2(n)$, leaving $O(n^{3/4})$ columns to test.

For each candidate $\mu$, the values $\chi^{\lambda}(\mu)$ were
computed recursively using the Murnaghan--Nakayama rule, with rim
hooks enumerated via beta-numbers and dimensions
$\chi^{\lambda}(1^m)$ computed by the branching rule.  The recursion
was memoized on the triple $(\lambda,a',b')$, which allows the memo
table to be shared across all candidate columns for a given $n$.  The
calculation was stopped as soon as a partition $\lambda$ with
$\chi^{\lambda}(\mu)=0$ was found; only genuinely zero-free columns
require checking all rows.  As sanity checks on the implementation,
the full character tables for $n\leq 8$ were verified against the
usual row and column orthogonality relations, and for $n\leq 14$ the
reduced search was verified against a brute-force search over all
$p(n)^2$ entries; the two agree for all $3\leq n\leq 14$, and the
unique discrepancy at $n=2$ is the exceptional transposition column of
Remark~\ref{rem:n2}.

Table~\ref{tab:Dn} records $D(n)$ for $1\leq n\leq 43$.

\begin{table}[ht]
\centering
\small
\begin{tabular}{c|ccccccccccccccc}
$n$    & 1 & 2 & 3 & 4 & 5 & 6 & 7 & 8 & 9 & 10 & 11 & 12 & 13 & 14 & 15\\
$D(n)$ & 1 & 2 & 2 & 2 & 1 & 1 & 4 & 1 & 1 & 1  & 4  & 1  & 4  & 1  & 2\\[2pt]
\hline\\[-8pt]
$n$    & 16 & 17 & 18 & 19 & 20 & 21 & 22 & 23 & 24 & 25 & 26 & 27 & 28 & 29 & 30\\
$D(n)$ & 2  & 1  & 2  & 2  & 3  & 1  & 3  & 2  & 2  & 2  & 3  & 2  & 1  & 3  & 1\\[2pt]
\hline\\[-8pt]
$n$    & 31 & 32 & 33 & 34 & 35 & 36 & 37 & 38 & 39 & 40 & 41 & 42 & 43 & &\\
$D(n)$ & 5  & 1  & 3  & 1  & 2  & 1  & 1  & 1  & 5  & 1  & 3  & 1  & 7  & &
\end{tabular}
\medskip
\caption{Values of $D(n)$ for $1\leq n\leq 43$.}
\label{tab:Dn}
\end{table}

In particular, the integers $n\leq 43$ with $D(n)=1$ are
$1$, $5$, $6$, $8$, $9$, $10$, $12$, $14$, $17$, $21$, $28$, $30$,
$32$, $34$, $36$, $37$, $38$, $40$, $42$, extending the list reported
by Duro \cite{Duro} by the entries $40$ and $42$.  As an illustration
of how nearly the bounds of Corollaries~\ref{cor:bbound}
and~\ref{cor:abound} can be attained, the seven zero-free columns of
$S_{43}$ (where $s_3(43)=1$ and $s_2(43)=11$) are
\[
  (1^{43}),\quad (2^4,1^{35}),\quad (2^8,1^{27}),\quad
  (3,2^2,1^{36}),\quad (3,2^4,1^{32}),\quad (3,2^6,1^{28}),\quad
  (3,2^{10},1^{20}).
\]


\begin{thebibliography}{99}

\bibitem{BC}
R.~P. Bambah and S.~Chowla,
\emph{On numbers which can be expressed as a sum of two squares},
Proc. Nat. Inst. Sci. India \textbf{13} (1947), 101--103.

\bibitem{Duro}
D.~Mart\'in Duro,
\emph{Arithmetic properties of character degrees and the generalised
Knutson index},
Ramanujan J. \textbf{69} (2026), article no.~37.

\bibitem{GO}
A.~Granville and K.~Ono,
\emph{Defect zero $p$-blocks for finite simple groups},
Trans. Amer. Math. Soc. \textbf{348} (1996), no.~1, 331--347.

\bibitem{HR}
G.~H. Hardy and S.~Ramanujan,
\emph{Asymptotic formulae in combinatory analysis},
Proc. London Math. Soc. (2) \textbf{17} (1918), 75--115.

\bibitem{JK}
G.~James and A.~Kerber,
\emph{The Representation Theory of the Symmetric Group},
Encyclopedia of Mathematics and its Applications, Vol.~16,
Addison--Wesley, Reading, MA, 1981.

\bibitem{Ma}
J.~U. Marshall,
\emph{The Loeschian numbers as a problem in number theory},
Geogr. Anal. \textbf{7} (1975), no.~4, 421--426.

\bibitem{MR}
K.~Matom\"aki and M.~Radziwi\l\l,
\emph{Multiplicative functions in short intervals II},
arXiv:2007.04290.

\bibitem{Mi}
A.~R. Miller,
\emph{The probability that a character value is zero for the symmetric
group},
Math. Z. \textbf{277} (2014), no.~3--4, 1011--1015.

\bibitem{MS}
A.~R. Miller and D.~Scheinerman,
\emph{Large-scale Monte Carlo simulations for zeros in character
tables of symmetric groups},
Math. Comp. \textbf{94} (2025), no.~351, 505--515.

\bibitem{Ol}
J.~B. Olsson,
\emph{Combinatorics and representations of finite groups},
Vorlesungen aus dem Fachbereich Mathematik der Universit\"at Essen,
Heft 20, 1993.

\bibitem{On}
K.~Ono,
\emph{On the positivity of the number of $t$-core partitions},
Acta Arith. \textbf{66} (1994), no.~3, 221--228.

\bibitem{PS}
S.~Peluse and K.~Soundararajan,
\emph{Zeros in the character table of the symmetric group},
arXiv:2603.28510.

\bibitem{Wi}
K.~S. Williams,
\emph{Mertens' theorem for arithmetic progressions},
J. Number Theory \textbf{6} (1974), 353--359.

\end{thebibliography}
\end{document}